\documentclass[11pt,reqno]{amsart}
\usepackage{amssymb,latexsym}
\usepackage[T2A]{fontenc}
 \usepackage[cp1251]{inputenc}
\usepackage[english,ukrainian]{babel}
\theoremstyle{plain}
\numberwithin{equation}{section}
\newtheorem{thm}{Theorem}
\newtheorem{lemma}{Lemma}[section]
\newtheorem{theorem}{Theorem}[section]
\newtheorem{definition}[thm]{Definition}

\begin{document}

\setcounter{page}{1}

\title[Derivations and identities ]{Derivations and identities for Fibonacci and Lucas polynomials}
\author{Leonid Bedratyuk}
\address{Department of Mathematics and Statistics\\
                Khmelnytskyi National  University\\
                Khmelnytskyi, Instytutska ,11\\
                29016, Ukraine}
\email{leonid.uk@gmail.com}

\begin{abstract}
We introduce the notion of Fibonacci and Lucas derivations of the polynomial algebras and prove that any element of kernel of the derivations defines a polynomial identity  for the Fibonacci and Lucas polynomials. Also, we  prove  that any polynomial identity  for Appel polynomial yields a polynomial  identity for the Fibonacci and Lucas polynomials та описуємо   arised interwining maps.
\end{abstract}

\maketitle

\section{Introduction}

The Fibonacci     $F_n(x)$  and Lucas   $L_n(x)$  polynomials are defined by the following ordinary generating functions
\begin{gather*}
\mathcal{G}(F_n(x),t)=\frac{t }{1-xt-t^2}=\sum_{n=0}^{\infty} F_n(x)  t^n,\\
 \mathcal{G}(L_n(x),t)=\frac{1+t^2  }{1-xt-t^2}=\sum_{n=0}^{\infty}L_n(x)  t^n. 
\end{gather*}

 The derivatives of the polynomials have the form , see  \cite{F_P},\cite{Filip}: 
 \begin{align}
&\frac{d}{dx} F_n(x)=\sum_{k=0}^{\left[\frac{n-1}{2} \right]} (-1)^k ({n{-}1{-}2k})F_{n{-}1{-}2k}(x),\label{der1}\\
&\frac{d}{dx} L_n(x)=n \sum_{k=0}^{\left[\frac{n-1}{2} \right]} (-1)^k L_{n{-}1{-}2k}(x),\label{der2}
 \end{align}
 
We are interested in finding  polynomial identities for the  polynomials, i.e.  identities  of the form
$$
P(F_0(x),F_1(x),\ldots, F_n(x))={\rm const} \,\,\text{  and }  P(L_0(x),L_1(x),\ldots, L_n(x))={\rm const},
$$
where  $P(x_0,x_1,\ldots,x_n)$ is a polynomial of  $n+1$ variables.
We offer a method for finding such identities which  is based on the followimng simple observation:
if 
$$
\frac{d}{dx} P(F_0(x),F_1(x),\ldots, F_n(x))=0,
$$
тоді 
$
P(F_0(x),F_1(x),\ldots, F_n(x))={\rm const}
$
i.e., it is an indentity.

On the other hand, rewrite  this derivative in the form 
\begin{gather*}
\frac{d}{dx} P(F_0(x),F_1(x),\ldots, F_n(x))=\\=\frac{\partial }{\partial x_0}P(x_0,x_1,\ldots,x_n)\Big |_{\{x_i=F_i(x)\}} \frac{d}{dx} F_0(x)+\cdots +\frac{\partial }{\partial x_n}P(x_0,x_1,\ldots,x_n)\Big |_{\{x_i=F_i(x)\}} \frac{d}{dx} F_n(x).
\end{gather*}
Therefore  
 
 \begin{gather*}
\frac{d}{dx} P(F_0(x),F_1(x),\ldots, F_n(x))=\\=\left( \frac{\partial }{\partial x_0}P(x_0,x_1,\ldots,x_n) \mathcal{D}_{\mathcal{F}}(x_0) +\cdots +\frac{\partial }{\partial x_n}P(x_0,x_1,\ldots,x_n)\mathcal{D}_{\mathcal{F}}(x_n)\right) \Big |_{\{x_i=F_i(x)\}} =\\=
\mathcal{D}_{\mathcal{F}}(P(x_0,x_1,\ldots,x_n))\Big |_{\{x_i=F_i(x)\}},
\end{gather*}
 where the differential operator  $\mathcal{D}_{\mathcal{F}}$ defined by 
 $$
 \mathcal{D}_{\mathcal{F}}:=x_1 \frac{\partial}{\partial x_2}{+}2 x_2 \frac{\partial}{\partial x_3}+(3\,x_{{3}}{-}x_{{1}})\frac{\partial}{\partial x_4}{+} \cdots+\bigl(\sum_{k=0}^{\left[\frac{n-1}{2} \right]} (-1)^k ({n{-}1{-}2k})x_{n{-}1{-}2k})\bigr) \frac{\partial}{\partial x_n}.
 $$
 It is clear that if $\mathcal{D}_{\mathcal{F}}(P(x_0,x_1,\ldots,x_n))$ then $\dfrac{d}{dx} P(F_0(x),F_1(x),\ldots, F_n(x))=0.$
 Thus, any non-trivial polynomial  $P(x_0,x_1,\ldots,x_n)$, which belongs to the kernel of   $\mathcal{D}_{\mathcal{F}}$, i.e., the following holds   ${ \mathcal{D}_{\mathcal{F}}(P(x_0,x_1,\ldots,x_n))=0}$,   defines a polynomial identity of the form  $P(F_0(x),F_1(x),\ldots, F_n(x))={\rm const }.$ Analogously, for the Lucas folynomials, we introduce the differential operator 
 $$
  \mathcal{D}_{\mathcal{L}}(x_n)=n \sum_{k=0}^{\left[\frac{n-1}{2} \right]} (-1)^k x_{n{-}1{-}2k}
 $$
and show  that    the condition $$ \mathcal{D}_{\mathcal{L}}(P(x_0,x_1,\ldots,x_n))=0$$  implies the identity  
$P(L_0(x),L_1(x),\ldots, L_n(x))={\rm const }.$
For instance, it is easily verified that for the polynomial  $x_1 x_3-x_2^2$ we have $\mathcal{D}_{\mathcal{F}}(x_1 x_3-x_2^2)=0$, thus  $F_{{1}}(x)F_{{3}}(x)-{F_{{2}}}(x)^{2}$ is a constant. The substitution $F_n(x)$  of $x_n$  gives
 \begin{gather*}
F_{{1}}(x)F_{{3}}(x)-{F_{{2}}}(x)^{2}=1.
\end{gather*}

A similar problem  was solved  in \cite{B_App}  for  the Appel polynomials. 
 In this paper was proved that any non-trivial element of kernel of the differential operator 
 $$
\mathcal{D}_{\mathcal{A}}=x_0 \frac{\partial}{\partial x_1}+2 x_1 \frac{\partial }{\partial x_2}+\cdots+ n x_{n-1} \frac{\partial }{\partial x_n}.
$$
 defines some  polynomial identity for the Appel polynomial.
 Recall that    polynomials   $\{A_n(x) \},$  $\deg(A_n(x))=n$ is called  the Appel polynomials if 
\begin{gather}
A'_n(x)=n A_{n-1}(x), n=0,1,2,\ldots .
\end{gather}

 In the present paper we show  how the   known polylomial identities for Appel can be used to find a polynomials identities for the Fibonacci and Lucas polynomials.  
   A linear map  $\psi_{AF}$   is called  $(\mathcal{D}_{\mathcal{A}},\mathcal{D}_{\mathcal{F}})$-interwin\-ing map if the following condition holds: $\psi_{AF} \mathcal{D}_{\mathcal{A}}=\mathcal{D}_{\mathcal{F}} \psi_{AF} $. Any such map induces  an ispmorphism from   $\ker \mathcal{D}_{\mathcal{A}}$  to  $\ker \mathcal{D}_{\mathcal{F}}.$ 
 For instance,  the discriminant of the polynomial (in the variables $X,Y$)
 
$$
x_{{0}}{X}^{3}+3\,x_{{1}}{X}^{2}Y+3\,x_{{2}}X{Y}^{2}+x_{{3}}{Y}^{3},
$$
equals
$$
\begin{vmatrix}
x_0 & 3 x_1 &3 x_2 & x_3 &0 \\
0 & x_0 & 3 x_1 &3 x_2 & x_3 \\
3x_0 & 6 x_1 &3 x_2 & 0 &0 \\
0&3x_0 & 6 x_1 &3 x_2 & 0  \\
0&0&3x_0 & 6 x_1 &3 x_2  \\
\end{vmatrix}=27 (6\,x_{{0}}x_{{3}}x_{{2}}x_{{1}}+3\,{x_{{1}
}}^{2}{x_{{2}}}^{2}-4\,{x_{{1}}}^{3}x_{{3}}-4\,{x_{{2}}}^{3}x_{{0}}-{x_{{0}}}^{2}{x_{{3}}}^{2}
),
$$
 and lies in the kernel of the operator $\mathcal{D}_{\mathcal{A}}$. It is  well known result of the classical invariant theory. 
It is easily checked that the linear map  defined by 
\begin{align*}
&\psi_{AL}(x_0)=x_0,
\psi_{AL}(x_1)=x_1,\\
&\psi_{AL}(x_2)=x_2,
\psi_{AL}(x_3)=x_3+3 x_1,
\end{align*}
commutes with the operators  $\mathcal{D}_{\mathcal{A}}$ and $\mathcal{D}_{\mathcal{L}}$. Therefore the element
$$
\begin{vmatrix}
\psi_{AL}(x_0) & 3 \psi_{AL}(x_1) &3 \psi_{AL}(x_2) & \psi_{AL}(x_3) &0 \\
0 & \psi_{AL}(x_0) & 3\psi_{AL}( x_1) &3 \psi_{AL}(x_2) & \psi_{AL}(x_3) \\
3\psi_{AL}(x_0) & 6 \psi_{AL}(x_1) &3 \psi_{AL}(x_2) & 0 &0 \\
0&3\psi_{AL}(x_0) & 6 \psi_{AL}(x_1) &3\psi_{AL}( x_2) & 0  \\
0&0&3\psi_{AL}(x_0) & 6 \psi_{AL}(x_1) &3 \psi_{AL}(x_2)  \\
\end{vmatrix}=\begin{vmatrix}
x_0 & 3 x_1 &3 x_2 & x_3+3 x_1 &0 \\
0 & x_0 & 3 x_1 &3 x_2 & x_3+3 x_1 \\
3x_0 & 6 x_1 &3 x_2 & 0 &0 \\
0&3x_0 & 6 x_1 &3 x_2 & 0  \\
0&0&3x_0 & 6 x_1 &3 x_2  \\
\end{vmatrix},
$$
lies in the kernel of the operator  $\mathcal{D}_{\mathcal{L}}$ and defines the following identity for the Lucas polynomial:

$$
\begin{vmatrix}
L_0(x) & 3 L_1(x) &3 L_2(x) & L_3(x)+3 L_1(x) &0 \\
0 & L_0(x) & 3 L_1(x) &3 L_2(x) & L_3(x)+3 L_1(x) \\
3L_0(x) & 6 L_1(x) &3 L_2(x) & 0 &0 \\
0&3L_0(x) & 6 L_1(x) &3 L_2(x) & 0  \\
0&0&3L_0(x) & 6 L_1(x) &3 L_2(x)  \\
\end{vmatrix}=-864.
$$

In the  paper we present methods of the theory of locally  nilpotent  derivation  to find  polynomial  identities for  the Fibonacci and Lucal polynomials.

In section 2,we give  a brief introduction to the theory of localy nilpotent derivations. Also, we  introduce the notion of the Fibonacci and Lucas derivations and find its kernels. In  this way we obtain  some  polynomials identities for the Fibonacci and Lucas polynomials.

In the section 3 we find a $(\mathcal{D}_{\mathcal{A}},\mathcal{D}_{\mathcal{F}})$-interwining map and a $(\mathcal{D}_{\mathcal{A}},\mathcal{D}_{\mathcal{L}})$-interwining map.


\section{Fibonacci and Lucas derivations }

\subsection{Derivations and its kernels} Let   $\mathbb{C}[x_0,x_1,x_2,\ldots,x_n]$ be the polynomial algebra  in  $n+1$ variables $x_0,x_1,x_2,\ldots,x_n$ over $\mathbb{C}.$ Recall  that a  {\it derivation} of the polynomial algebra $\mathbb{C}[x_0,x_1,x_2,\ldots,x_n]$ is a linear  map  $D$ satisfying the Leibniz rule: 
$$
D(x_1 \, x_2)=D(x_1) x_2+x_1 D(x_2), \text{  for all }  x_1, x_2 \in \mathbb{C}[x_0,x_1,x_2,\ldots,x_n].
$$
A derivation $D$  is called {\it locally nilpotent} if for every $f \in \mathbb{C}[x_0,x_1,x_2,\ldots,x_n]$ there is an $n \in \mathbb{N}$ such that $D^n(f)=0.$ 
Any derivation   $D$ is completely determined by the elements $D(x_i).$ A  derivation   $D$  is called  \textit{linear} if  $D(x_i)$ is a linear form. A  linear locally nilpotent derivation is called a \textit{Weitzenb\"ock derivation}. 
A derivation    $D$ 
is called the triangular if  $D(x_i) \in \mathbb{C}[x_0,\ldots,x_{i-1}].$ Any triangular derivation is locally nilpotent.

The subalgebra 
$$
\ker D:=\left \{ f \in \mathbb{C}[x_0,x_1,x_2,\ldots,x_n] \mid  D(f)=0 \right \},
$$
is called the {\it kernel} of derivation $D.$

For arbitrary localy niplotent derivation  $D$ tye following statement holds:
\begin{theorem} \label{maint} Suppose that there exists  a polynomials $h$ such that  $D(h) \neq 0$ but $D^2(h)=0.$ Then
$$
\ker D=\mathbb{C}[\sigma(x_0),\sigma(x_2),\ldots,\sigma(x_n)][D(h)^{-1}] \cap \mathbb{C}[x_0,x_1,\ldots,x_n],
$$
where $\sigma$ is the Diximier map 
$$
\sigma(x_i)=\sum_{k=0}^{\infty} D^k(x_i) \frac{\lambda^k}{k!},\lambda=-\dfrac{h}{D(h)}, D(\lambda)=-1.
$$
\end{theorem}
The  proof one may find   in \cite{Now} and \cite{Essen}. 

The expressions  (\ref{der1}),(\ref{der2}) motivate the following definition  
\begin{definition}
Derivations of  $\mathbb{C}[x_0,x_1,x_2,\ldots,x_n]$  defined by 
 \begin{align*}
&D_\mathcal{F}(x_0)=D_\mathcal{F}(x_1)=0, D_\mathcal{F}(x_n)=\sum_{k=0}^{\left[\frac{n-1}{2} \right]} (-1)^k ({n{-}1{-}2k})x_{n{-}1{-}2k},\\
&D_\mathcal{L}(x_0)=0, D_\mathcal{L}(x_n)=n \sum_{k=0}^{\left[\frac{n-1}{2} \right]} (-1)^k x_{n{-}1{-}2k}, i=2,3,\ldots, n, \ldots.,
\end{align*}
are called  {\bf the Fibonacci derivation} and {\bf the Lucas derivation}  respectivelly
\end{definition}
We  have
$$
\begin{array}{ll}
D_\mathcal{F}(x_0)=0,& D_\mathcal{L}(x_0)=0,\\
D_\mathcal{F}(x_1)=0,&D_\mathcal{L}(x_1)=x_0,\\
 D_\mathcal{F}(x_2)=x_1,& D_\mathcal{L}(x_2)=2x_1\\
 D_\mathcal{F}(x_3)=2x_2,& D_\mathcal{L}(x_3)=3\,x_{{2}}-3\,x_{{0}},\\
 D_\mathcal{F}(x_4)=3\,x_{{3}}-x_{{1}},& D_\mathcal{L}(x_4)=4\,x_{{3}}-4\,x_{{1}},\\
 D_\mathcal{F}(x_5)=4\,x_{{4}}-2\,x_{{2}},& D_\mathcal{L}(x_5)=5\,x_{{4}}-5\,x_{{2}}\\
 D_\mathcal{F}(x_6)=5\,x_{{5}}-3\,x_{{3}}+x_{{1}}, & D_\mathcal{L}(x_6)=6\,x_{{5}}-6\,x_{{3}}+6x_{{1}}.
\end{array}
$$

We define  the substitution homomorphisms
 $\varphi_{\mathcal{F}},\varphi_{\mathcal{L}}:\mathbb{C}[x_0,x_1,\ldots,x_n] \to \mathbb{C}[x]  $  by   $\varphi_{\mathcal{F}}(x_i)=F_i(x)$ and by 
$\varphi_{\mathcal{L}}(x_i)=L_i(x)$.
Put
\begin{gather*}
\ker \varphi_{\mathcal{F}}:=\{ P ( x_0,x_1,...,x_n)  \mid \varphi_{\mathcal{F}}(P( x_0,x_1,...,x_n)) =0 \},\\
\ker \varphi_{\mathcal{L}}:=\{ P( x_0,x_1,...,x_n)  \mid \varphi_{\mathcal{L}}\,(P( x_0,x_1,...,x_n)) =0 \}.
\end{gather*}
Any element of $\ker \varphi_{\mathcal{F}}$  or $\ker \varphi_{\mathcal{F}}$  defines a polynomial identity for the Fibonacci and Lucas polynomials відповідно. Put 
\begin{gather*}
\ker \mathcal{D}_{\mathcal{F}}:=\{ S \in  \mathbb{C}[ x_0,x_1,...,x_n]  \mid \mathcal{D}_{\mathcal{F}}(S) =0 \},\\
\ker \mathcal{D}_{\mathcal{L}}:=\{ S \in  \mathbb{C}[ x_0,x_1,...,x_n]  \mid \mathcal{D}_{\mathcal{L}}(S) =0 \} .
\end{gather*}
It is  easy to see that   $\varphi_{\mathcal{F}}  \mathcal{D}_{\mathcal{F}} =\dfrac{d}{dx}\, \varphi_{\mathcal{F}}$  and $\varphi_{\mathcal{L}}  \mathcal{D}_{\mathcal{L}} =\dfrac{d}{dx}\, \varphi_{\mathcal{L}}$.  It follows that $\varphi_{\mathcal{F}} (\ker \mathcal{D}_{\mathcal{F}})\subset \ker \varphi_{\mathcal{F}}$  and $\varphi_{\mathcal{L}} (\ker \mathcal{D}_{\mathcal{L}})\subset \ker \varphi_{\mathcal{L}}$. 
Note that $\varphi_{\mathcal{L}} (\ker \mathcal{D}_{\mathcal{L}}) \neq  \ker \varphi_{\mathcal{L}}$. In fact, we  have  $\varphi_{\mathcal{F}}(x_n-x_2 x_{n-1}+x_{n-2})=F_n(x)-(F_2(x) F_{n-1}(x)+F_{n-2}(x))=0$  but $x_n-x_2 x_{n-1}-x_{n-2} \notin \ker \mathcal{D}_{\mathcal{F}}.$


We have thus proved the following theorem. 
\begin{theorem}  Let  $P(x_0,x_1,\ldots,x_n)$ be a polynomial. 

(i) If $\mathcal{D}_{\mathcal{F}}(P(x_0,x_1,\ldots,x_n))=0$ then  $P(F_0(x),F_1(x),\ldots, F_n(x))=\rm{const}$;

(ii) if  $\mathcal{D}_{\mathcal{L}}(P(x_0,x_1,\ldots,x_n))=0$ then $P(L_0(x),L_1(x),\ldots, L_n(x))=\rm{const}$.
\end{theorem}


\subsection{The kernel of the Fibonacci derivation}

It is obviously that this derivation is triangular and hence   localy nilpotent. Thus to find its kernel we may use  Theorem  \ref{maint}.

Let us construct  the Diximier map for the Fibonacci derivation. For this purpose, we derive first a close expression for  $D^k_\mathcal{F}(x_n).$

We  have
\begin{gather*}
D^2_\mathcal{F}(x_n)=D_\mathcal{F}\left(\sum_{k=0}^{\left[\frac{n-1}{2} \right]} (-1)^k ({n{-}1{-}2k})x_{n{-}1{-}2k}\right)=\\=
\sum_{k=0}^{\left[\frac{n-1}{2} \right]} (-1)^k ({n{-}1{-}2k})\sum_{j=0}^{\left[\frac{n-2-2k}{2} \right]} (-1)^j ({n{-}2{-}2k}-2j)x_{n{-}2{-}2j}=
\\=\sum_{i=0}^{\left[\frac{n-2}{2} \right]}  (-1)^{i} (i+1) (n-2-2i) (n-(i+1))x_{n-2-2i}.
\end{gather*}
Similarly we  get
\begin{gather*}
D^3_\mathcal{F}(x_n)=\sum_{i=0}^{\left[\frac{n-3}{2} \right]}  (-1)^i \frac{(i+1)(i+2)}{2} (n-3-2i)(n-(i+1))(n-(i+2))x_{n-3-2i}.
\end{gather*}
Induction gives
\begin{gather*}
D^k_\mathcal{F}(x_n)=(k-1)! \sum_{i=0}^{\left[\frac{n-k}{2} \right]}  (-1)^i (n-k-2i) {i+k-1 \choose k-1} {n-i-1 \choose k-1} x_{n-k-2i}.
\end{gather*}
Since  $D_\mathcal{F}\left(-\displaystyle  \frac{x_2}{x_1} \right)=-1$  we put $\lambda=-\displaystyle  \frac{x_2}{x_1}.$ Now we may find the Diximier map:
\begin{gather*}
\sigma(x_n)=\sum_{k=0}^{n-1}D^k_\mathcal{F}(x_n) \frac{\lambda^k}{k!}=\\
=x_n+\sum_{k=1}^{n-1} \frac{\lambda^k}{k} \sum_{i=0}^{\left[\frac{n-k}{2}\right]}  (-1)^{k+i} (n-k-2i) {i+k-1 \choose k-1} {n-i-1 \choose k-1} x_{n-k-2i}=\\=
x_n+\sum_{k=1}^{n-3} \frac{\lambda^k}{k} \sum_{i=0}^{\left[\frac{n-k}{2}\right]}  (-1)^{i} (n{-}k{-}2i) {i+k-1 \choose k-1} {n-i-1 \choose k-1} x_{n-k-2i}+(n-1)x_2 \lambda^{n-2}+x_1 \lambda^{n-1}.
\end{gather*}
Replasing  $\lambda$ by $-\displaystyle  \frac{x_2}{x_1} $, we obtain, after   simplifying:
\begin{multline*}
x_1^{n-2} \sigma(x_n)=\\=  x_n x_1^{n-2}+\sum_{k=1}^{n-3} \frac{1}{k} \sum_{i=0}^{\left[\frac{n-k}{2}\right]}   (-1)^{k+i} (n{-}k{-}2i) {i{+}k{-}1 \choose k{-}1} {n{-}i{-}1 \choose k{-}1} x_{n{-}k{-}2i} x_2^k x_1^{n{-}2{-}k}+(n-2) (-1)^{n-2}x_2^{n-1}.
\end{multline*}
The polynomials
\begin{multline*}
C_n:=\\=x_n x_1^{n-2}+\sum_{k=1}^{n-3} \frac{1}{k} \sum_{i=0}^{\left[\frac{n-k}{2}\right]}   (-1)^{k+i} (n{-}k{-}2i) {i{+}k{-}1 \choose k{-}1} {n{-}i{-}1 \choose k{-}1} x_{n{-}k{-}2i} x_2^k x_1^{n{-}2{-}k}{+}\\+(n{-}2) (-1)^{n{-}2}x_2^{n-1}, n>2,
\end{multline*}
belong to the kernel $\ker \mathcal{D}_{\mathcal{F}}.$ We call them   \textit{the Cayley elements} of the localy nilpotent derivation $\mathcal{D}_{\mathcal{F}}$.  
The first few Cayley polynomials are shown below:
\begin{align*}
&C_3=-{x_{{2}}}^{2}+x_{{3}}x_{{1}},\\
&C_4=2\,{x_{{2}}}^{3}-3\,x_{{2}}x_{{3}}x_{{1}}+{x_{{1}}}^{2}x_{{2}}+x_{{4}}
{x_{{1}}}^{2},\\
&C_5=-3\,{x_{{2}}}^{4}+6\,{x_{{2}}}^{2}x_{{3}}x_{{1}}-{x_{{1}}}^{2}{x_{{2}}
}^{2}-4\,x_{{2}}x_{{4}}{x_{{1}}}^{2}+x_{{5}}{x_{{1}}}^{3},\\
&C_6=4\,{x_{{2}}}^{5}-10\,{x_{{2}}}^{3}x_{{3}}x_{{1}}-2\,{x_{{1}}}^{2}{x_{{
2}}}^{3}+10\,{x_{{2}}}^{2}x_{{4}}{x_{{1}}}^{2}-5\,x_{{2}}x_{{5}}{x_{{1
}}}^{3}+3\,{x_{{1}}}^{3}x_{{2}}x_{{3}}-{x_{{1}}}^{4}x_{{2}}+x_{{6}}{x_
{{1}}}^{4}.
\end{align*}

Theorem \ref{maint} implies

\begin{theorem}
\begin{align*}
&\ker {\mathcal{D}_{\mathcal{F}}}=\mathbb{C}[x_0,x_1,C_3,C_4,\ldots,C_n][x_1^{-1}] \cap \mathbb{C}[x_0,x_1,\ldots,x_n].
\end{align*}
\end{theorem}

Thus, we obtain  a description of the kernel of the Fibonacci derivation.

To get an identity for the Fibonacci polynomials we should find  $\varphi_F(C_n).$  We  get 
\begin{multline*}
\varphi_F(C_n)=\varphi_F(C_n(x_0,x_1,\ldots,x_n))=
F_n(x)+\\+\sum_{k=1}^{n-3} \frac{1}{k} \sum_{i=0}^{\left[\frac{n-k}{2}\right]}   (-1)^{k+i} (n{-}k{-}2i) {i{+}k{-}1 \choose k{-}1} {n{-}i{-}1 \choose k{-}1} F_{n{-}k{-}2i}(x) F_2(x)^k {+}\\+(n{-}2) (-1)^{n{-}2}F_2(x)^{n-1}={\rm const }.
\end{multline*}

\textbf{Conjecture.} $\varphi_F(C_n)=C_n(F_1(x),\ldots,F_n(x))=\left\{ \begin{array}{ll} 0, &  n  \text{ even},\\ 1, & n \text{ odd}. \end{array} \right.$

\subsection{Ядро диференціювання Люка} Для  диференціювання   $D_\mathcal{L}$    отримаємо схожим способом 
\begin{gather*}
D^k_\mathcal{L}(x_n)=n\, (k-1)!  \sum_{i=0}^{\left[\frac{n-k}{2}\right]}    (-1)^i {i+k-1 \choose k-1} {n-i-1 \choose k-1} x_{n-k-2i}.
\end{gather*}

Маємо, що $D_\mathcal{L}\left(-\displaystyle  \frac{x_1}{x_0} \right)=-1.$  Шукаємо  відображення Діксм'є:
\begin{gather*}
\sigma(x_n)=\sum_{k=0}^{n}  D_\mathcal{L}(x_n) \frac{\lambda^i}{k!}=\\
=x_n+n\,\sum_{k=1}^{n-2} \frac{\lambda^k}{k} \sum_{n-k-2i>0}  (-1)^{k+i}  {i+k-1 \choose k-1} {n-i-1 \choose k-1} x_{n-k-2i}+n x_1 \lambda^{n-1}+x_0 \lambda^n
\end{gather*}
Відповідний  елемент Келлі має вигляд
 $$
\sigma_n(x_n)=x_n x_0^{n-1}+n\,\sum_{k=1}^{n-2} \frac{1}{k} \sum_{i=0}^{\left[\frac{n-k}{2}\right]}  (-1)^{k+i}  {i+k-1 \choose k-1} {n-i-1 \choose k-1} x_{n-k-2i} x_1^k x_0^{n-1-k}+(n-1) (-1)^{n-1} x_1^n,
$$

\begin{theorem}
\begin{align*}
&k(x_0,x_1,\ldots,x_n)^{\mathcal{D}_{\mathcal{L}}}=k[x_0,C_2,C_3,C_4,\ldots,C_n][x_0^{-1}],\\
&k[x_0,x_1,\ldots,x_n]^{\mathcal{D}_{\mathcal{L}}}=k[x_0,C_2,C_3,C_4,\ldots,C_n][x_0^{-1}] \cap k[x_0,x_1,\ldots,x_n].
\end{align*}
\end{theorem}

Випишемо після  спрощення  декілька цих многочленів
\begin{align*}
&C_1=x_0,\\
&C_2=x_{{2}}x_{{0}}-{x_{{1}}}^{2},\\
&C_3=2\,{x_{{1}}}^{3}+3\,x_{{1}}{x_{{0}}}^{2}-3\,x_{{1}}x_{{2}}x_{{0}}+x_{{
3}}{x_{{0}}}^{2}
,\\
&C_4=-3\,{x_{{1}}}^{4}-4\,{x_{{1}}}^{2}{x_{{0}}}^{2}+6\,{x_{{1}}}^{2}x_{{2}
}x_{{0}}-4\,x_{{1}}x_{{3}}{x_{{0}}}^{2}+x_{{4}}{x_{{0}}}^{3}
,\\
&C_5=4\,{x_{{1}}}^{5}+10\,{x_{{1}}}^{2}x_{{3}}{x_{{0}}}^{2}-5\,x_{{1}}{x_{{0
}}}^{4}-5\,x_{{1}}x_{{4}}{x_{{0}}}^{3}-10\,{x_{{1}}}^{3}x_{{2}}x_{{0}}
+5\,x_{{1}}x_{{2}}{x_{{0}}}^{3}+x_{{5}}{x_{{0}}}^{4}
.
\end{align*} 
\textbf{Гіпотеза.} $C_n(L_0(x),\ldots,L_n(x))=\left\{ \begin{array}{ll} 2, &  n - \text{ парне}\\ 0, & \text{ інакше} \end{array} \right.$


\section{Appel-Lucas and Appel-Fibonacci    interwining maps}



\subsection{Appel-Lucas  interwinning map}


Denote       by  $\psi_{AL}$  a Appel-Lucas  intewinning map. Будемо шукати його у  вигляді:
$$
\psi_{AL}(x_n)=x_n+\alpha^{(1)}_n x_{n-2}+\alpha^{(2)}_n x_{n-4}+\ldots+\alpha^{(i)}_n x_{n-2i}+\ldots+\alpha_n^{\left( \left[ \frac{n-1}{2} \right]\right)} x_{n-2 \left[ \frac{n-1}{2} \right]}.
$$
Доведемо  таке твердження
\begin{lemma}
Послідовності  $\alpha^{(1)}_n,$ $\alpha^{(2)}_n,\ldots, \alpha_n^{\left( \left[ \frac{n-1}{2} \right]\right)}$  задовільняють такій системи рекурентних рівнянь 
$$
\left \{
\begin{array}{l} 
 (n-2) \alpha_n^{(1)}=n(\alpha_{n-1}^{(1)}+1),\\
(n-4) \alpha_n^{(2)}=n(\alpha_{n-1}^{(2)}+\alpha_{n-1}^{(1)}),\\
(n-2i) \alpha_n^{(i)}=n(\alpha_{n-1}^{(i)}+\alpha_{n-1}^{(i-1)}),\\
.................................................. \\
(n-2\left[\frac{n-1}{2} \right]) \alpha_n^{\left(\left[\frac{n-1}{2} \right]\right)}=n \left(\alpha_{n-1}^{\left(\left[\frac{n-1}{2} \right]\right)}+\alpha_{n-1}^{\left(\left[\frac{n-1}{2} \right]-1\right)}\right),
\end{array}
\right.
$$
із початковими  умовами $\alpha_{2i}^{(i)}=0.$
\end{lemma}
\begin{proof}
We have 
\begin{gather*}
D_\mathcal{L}(\psi_{AL}(x_n))=n(x_{n-1}-x_{n-3}+x_{n-5}-x_{n-7}+\ldots)+\alpha^{(1)}_n (n-2)(x_{n-3}-x_{n-5}+\ldots)+\\+\alpha^{(2)}_n (n-4)(x_{n-5}-x_{n-7}+\ldots)+\ldots
\end{gather*}
Звідси знаходимо
\begin{gather*}
D_\mathcal{L}(\psi_{AL}(x_n))=n x_{n-1}+x_{n-3}((n-2)\alpha^{(1)}_n -n)+x_{n-5}((n-4)\alpha^{(2)}_n-(n-2) \alpha^{(1)}_n+n)+\\+
x_{n-2i-1}((n-2i)\alpha^{(i)}_n-(n-2(i-1))\alpha^{(i-1)}_n+(n-2(i-2))\alpha^{(i-2)}_n+\ldots)+\ldots+
\end{gather*}
З іншого боку,  має виконуватися тотожність,
$$
D_\mathcal{L}(\psi_{AL}(x_n))=\psi_{AL}(\mathcal{D}(x_n))=n \psi_{AL}(x_{n-1})=n \sum_{n-1-2i>0} \alpha^{(i)}_{n-1} x_{n-1-2i}(-1)^i.
$$
Прирівнюючи відповідні коефіцієнти, отримуємо  такі рекурентні  співвідношення на  $\alpha^{(i)}_n:$
\begin{gather*}
(n-2)\alpha^{(1)}_n -n=n \alpha^{(1)}_{n-1}, \alpha^{(1)}_{2}=0,\\
(n-4)\alpha^{(2)}_n-(n-2) \alpha^{(1)}_n+n=n \alpha^{(2)}_{n-1},\alpha_4^{(1)}=0,\\
\ldots\\
\sum_{n-2i>0} (n-2i) \alpha^{(i)}_n (-1)^i =n \alpha^{(i)}_{n-1}, \alpha_{2i}^{(i)}=0.
\end{gather*}
Після  спрощення  отримуємо  систему рекурентних рівнянь
$$
\left \{
\begin{array}{l} 
 (n-2) \alpha_n^{(1)}=n(\alpha_{n-1}^{(1)}+1),\\
(n-4) \alpha_n^{(2)}=n(\alpha_{n-1}^{(2)}+\alpha_{n-1}^{(1)}),\\
(n-2i) \alpha_n^{(i)}=n(\alpha_{n-1}^{(i)}+\alpha_{n-1}^{(i-1)}),\\
.................................................. \\
(n-2\left[\frac{n-1}{2} \right]) \alpha_n^{\left(\left[\frac{n-1}{2} \right]\right)}=n \left(\alpha_{n-1}^{\left(\left[\frac{n-1}{2} \right]\right)}+\alpha_{n-1}^{\left(\left[\frac{n-1}{2} \right]-1\right)}\right),
\end{array}
\right.
$$
на послідовності $\alpha_{n}^{(1)},\alpha_{n}^{(2)},\ldots, \alpha_{n}^{\left(\left[\frac{n-1}{2} \right]\right)}$ із початковими умовами $\alpha_{2i}^{(i)}=0.$

\end{proof}
Розглянемо  допоміжне неоднорідне рекурентне співвідношення 
$$
(n-a) x_n=n(x_{n-1}+g_{n-1}), x_a=0,n \geq a.
$$
а  $g_n$ -- деяка  фіксована послідовність.  Тоді
\begin{lemma}
$$x_n=n^{{\underline{a}}}\sum_{i=a}^{n-1} \frac{g_i}{i^{\phantom{}^{\underline{a}}}}={n \choose a} \sum_{i=a}^{n-1} \dfrac{g_i}{{i \choose a}} $$
де $n^{\underline{a}}:=n(n-1)(n-2)\cdots (n-(a-1))$
\end{lemma}
\begin{proof}
Справді,  нехай
$$
x_n=n^{{\underline{a}}}\sum_{i=a}^{n-1} \frac{g_i}{i^{\phantom{}^{\underline{a}}}}.
$$ 
Тоді
\begin{gather*}
n(x_{n-1}+g_{n-1})=n\left( (n-1)^{{\underline{a}}}\sum_{i=a}^{n-2} \frac{g_i}{i^{\phantom{}^{\underline{a}}}}+g_{n-1}\right)=\\=n\left( (n-1)^{{\underline{a}}}\sum_{i=a}^{n-2} \frac{g_i}{i^{\phantom{}^{\underline{a}}}}+\frac{(n-1)^{{\underline{a}}}\,g_{n-1}}{(n-1)^{{\underline{a}}}}\right)=n(n-1)^{{\underline{a}}} \left( \sum_{i=a}^{n-2} \frac{g_i}{i^{\phantom{}^{\underline{a}}}}+\frac{g_{n-1}}{(n-1)^{{\underline{a}}}}\right)=\\=n(n-1)^{{\underline{a}}}\sum_{i=a}^{n-1} \frac{g_i}{i^{\phantom{}^{\underline{a}}}}=(n-a) n^{{\underline{a}}}\sum_{i=a}^{n-1} \frac{g_i}{i^{\phantom{}^{\underline{a}}}}=(n-a) x_n.
\end{gather*}
\end{proof}
Як  наслідок  маємо, що система  рекурентних рівнняь має такий  розвязок
\begin{gather*}
\alpha_n^{(1)}=n(n-2),\\
\alpha_n^{(k)}=(n)_{2k-1}\sum_{i=2k}^{n-1} \frac{\alpha^{(k-1)}_i}{(i)_{2k-1}}.
\end{gather*}

Розв'язуючи  систему послідовно знаходимо
\begin{align*}
&\alpha^{(1)}_n=n(n-2),\\
&\alpha^{(2)}_n=\frac14\, \left( n-1 \right) n \left( 3\,n-7 \right)  \left( n-4 \right)=\frac{1}{2!} (n-4) { n \choose 2} (3n-7),\\
&\alpha^{(3)}_n=\frac{1}{3!} (n-6) { n \choose 3}   \left( 19\,{n}^{2}-141\,n+254 \right),\\
&\alpha^{(4)}_n=\frac{1}{4!} (n-8) { n \choose 4}  \left( 
211\,{n}^{3}-3258\,{n}^{2}+16481\,n-27306 \right),\\
&\alpha^{(5)}_n=\frac{1}{5!} (n-10) { n \choose 5}  \left( 3651\,{n}^{4}-96550\,{n}^{3}+946185\,{n}^{2}-4071950\,n+
6492024 \right) 
\end{align*}
Для  знаходження загального розв'язку  системи 
розглянемо невідомі  послідовності $\alpha_n^{(s)}$ у  базисі із спадних факторіалів.  Нехай 
$$
\alpha_n^{(s)}=\beta_0^{(s)} n^{\underline{s}}+\beta_1^{(s)} n^{\underline{s+1}}+\cdots+\beta_s^{(s)} n^{\underline{2s}},
$$
  
Легко  бачити, що  
$$
(n-s)n^{\underline{s}}=n^{\phantom{}^{\underline{s+1}}}, n(n-1)^{\underline{s}}=n^{\phantom{}^{\underline{s+1}}}.  
$$
Тоді
\begin{gather*}
(n-2s)\alpha_n^{(s)}=(n-2s)(\beta_0^{(s)} n^{\underline{s}}+\beta_1^{(s)}n^{\underline{s+1}}+\cdots+\beta_s^{(s)}n^{\underline{2s}})=\\=((n-s)-s)\beta_0^{(s)} n^{\phantom{}^{\underline{s}}}+((n-(s+1))-(s-1)))\beta_1^{(s)}n^{\underline{s+1}}+\cdots+(n-2s)\beta_s^{(s)}n^{\underline{2s}}=\\=
\sum_{i=0}^{s}\beta_{i}^{(s)}n^{\phantom{}^{\underline{s+i+1}}}-\sum_{i=0}^{s-1}(s-i)\beta_{i}^{(s)}n^{\phantom{}^{\underline{s+i}}}
\end{gather*}
З іншого боку
\begin{gather*}
n(\alpha_{n-1}^{(s)}+\alpha_{n-1}^{(s-1)})=\sum_{i=0}^{s} n \beta_{i}^{(s)}(n-1)^{\phantom{}^{\underline{s+i}}}+\sum_{i=0}^{s-1} n \beta_{i}^{(s-1)}(n-1)^{\phantom{}^{\underline{s-1+i}}}=\sum_{i=0}^{s}  \beta_{i}^{(s)}n^{\phantom{}^{\underline{s+i+1}}}+\sum_{i=0}^{s-1}  \beta_{i}^{(s-1)}n^{\phantom{}^{\underline{s+i}}}.
\end{gather*}
Звідси, прирівнявши відповіднi  коефіцієнти, зразу знаходимо, що 
$$
\beta_i^{(s)}=\frac{\beta_{i}^{(s-1)}}{i-s},i=0\ldots s-1.
$$
Коефіцієнт $\beta_s^{(s)}$  знайдемо з умови, що $a_{2s}^{(s)}=0.$ Маємо
$$
\sum_{i=0}^{s-1}\frac{\beta_i^{(s-1)}}{i-s}(2s)^{\phantom{}^{\underline{s+i}}}+\beta_{s}^{(s)} (2s)!=0
$$
Звідси 
$$
\beta_{s}^{(s)}=\frac{1}{(2s)!}\sum_{i=0}^{s-1}\frac{\beta_i^{(s-1)}}{s-i}(2s)^{\phantom{}^{\underline{s+i}}}=\sum_{i=0}^{s-1} \frac{\beta_{i}^{(s-1)}}{(s-i) (s-i)!}=-\sum_{i=0}^{s-1} \frac{\beta_{i}^{(s)}}{ (s-i)!}.
$$
Таким чином, маємо  такi рекурентнi співвідношення  для $\beta_{n}^{(s)}$:
\begin{align*}
&\beta_0^{(s)}=-\frac{\beta_0^{(s-1)}}{s},\\
&\beta_1^{(s)}=-\frac{\beta_1^{(s-1)}}{s-1},\\
&\beta_2^{(s)}=-\frac{\beta_1^{(s-1)}}{s-2},\\
&\ldots \\
&\beta_{s-1}^{(s)}=-\beta_1^{(s-1)},\\
&\beta_{s}^{(s)}:=b_s=-\sum_{i=0}^{s-1} \frac{\beta_{i}^{(s)}}{ (s-i)!}.
\end{align*}

Звідси знаходимо, що 
$$
\beta_i^{(s)}=\frac{(-1)^{s-i}}{(s-i)!}\, b_i, \text{  for $i=0,\ldots,s-1$, }
$$
Oтримуємо таке рекурентне співвідношення на послідовність $b_s:$
$$
b_s=-\frac{(-1)^{s}}{(s!)^2}-\sum_{i=1}^{s-1} \frac{(-1)^{s-i}}{((s-i)!)^2}\, b_i,
$$
Поклавши $b_0=1$ отримаємо, що елементи послідовності $b_n$ задовільняють таке рекурентне співвідношення
$$
\sum_{i=0}^n \frac{(-1)^{n-i}}{(n-i)!^2}\,b_n=0,n>0
$$
Розглянемо ряд 
$$
\sum_{n=0}^{\infty} \frac{(-1)^{n}}{n!^2}z^n=J_0(\sqrt{4z}\,)
$$
і  звичайну породжуючу функцію 
$$
G(b_n,z)=\sum_{n=0}^{\infty} b_n x^n.
$$
Тоді, враховуючи рекурентне співвідношення знайдемо, що $$ J_0(\sqrt{4z}\,)G(b_n,z)=1.$$
Звідси 
$$
G(b_n,z)=\frac{1}{J_0(\sqrt{4z}\,)}=1+z+{\frac {3}{4}}{z}^{2}+{\frac {19}{36}}{z}^{3}+{\frac {211}{576}}{z}^{4}+{\frac {1217}{4800}}{z}^{5}+{\frac {30307}{172800}}{z}^{6}+\cdots
$$

Отже   ми  довели таку  теорему

\begin{theorem}
A Appel-Lucas   interwining map has  the form 
$$
\psi_{AL}(x_n)=x_n+\alpha^{(1)}_n x_{n-2}+\alpha^{(2)}_n x_{n-4}+\ldots+\alpha^{(i)}_n x_{n-2i}+\ldots+\alpha_n^{\left( \left[ \frac{n-1}{2} \right]\right)} x_{n-2 \left[ \frac{n-1}{2} \right]},
$$
where
$$
\alpha^{(s)}_n=\frac{(-1)^{s}}{s!} b_0 n^{{\underline{s}}}+\cdots+\frac{(-1)^{s-i}}{(s-i)!}\, b_i n^{\underline{s+i}}+\cdots+ b_s n^{\underline{2s}},
$$
and the generating function for  $b_0, b_1, \ldots, b_n,\ldots$  are defined через обернену  функцію Бесселя
$$
\sum_{i=0}^{\infty}b_i z^i=J_0^{-1}(\sqrt{4z}\,).
$$
\end{theorem}


\subsection{Appel-Fibonacci interwining map}


Шукаємо явний  вигляд ізоморфізму ядер $\ker \mathcal{D}_{\mathcal{A}} \stackrel{\psi}{\rightarrow} \ker \mathcal{D}_{\mathcal{F}}.$

Шукаємо $\psi$  у  вигляді:
$$
\psi(x_n)=\alpha^{(0)}_n x_{n+1}+\alpha^{(1)}_n x_{n-1}+\alpha^{(2)}_n x_{n-3}+\ldots+\alpha^{(i)}_n x_{n+1-2i}+\ldots+\alpha_n^{\left( \left[ \frac{n-1}{2} \right]\right)} x_{n+1-2 \left[ \frac{n-1}{2} \right]}.
$$
Доведемо твердження 

\begin{lemma}
Коефіцієнти $\alpha^{(1)}_n,$ $\alpha^{(2)}_n,\ldots, \alpha_n^{\left( \left[ \frac{n-1}{2} \right]\right)}$  задовільняють такій системи рекурентних рівнянь 
$$
\left\{ 
\begin{array}{l}
\alpha_{n}^{(0)}=1 ,\\
\alpha_{n}^{(1)}=n \left( \dfrac{\alpha_{n-1}^{(1)}}{n-2}+ \dfrac{\alpha_{n-1}^{(0)}}{n} \right),\\
\alpha_{n}^{(2)}=n \left( \dfrac{\alpha_{n-1}^{(2)}}{n-4}+ \dfrac{\alpha_{n-1}^{(1)}}{n-2} \right),\\
.................................................. \\
\alpha_{n}^{(s)}=n \left( \dfrac{\alpha_{n-1}^{(s)}}{n-2s}+ \dfrac{\alpha_{n-1}^{(s-1)}}{n-2(s-1)} \right),\\
\ldots
\end{array}
\right.
$$
із початковими  умовами $\alpha_{2i}^{(i)}=0.$
\end{lemma}

\begin{proof}
Тоді
\begin{gather*}
D_\mathcal{F}(\psi(x_n))=\alpha_{n}^{(0)}(n x_{n}-(n-2)x_{n-2}+(n-4)x_{n-4}-(n-6)x_{n-6}+\ldots)+\\+\alpha^{(1)}_n ((n-2)x_{n-2}-(n-4)x_{n-4}+\ldots)+\alpha^{(2)}_n ((n-4)x_{n-4}-(n-6)x_{n-6}+\ldots)+\ldots
\end{gather*}
Звідси знаходимо
\begin{gather*}
D_\mathcal{F}(\varphi(x_n))=n\alpha_{n}^{(0)} x_n+(n-2) x_{n-2}(\alpha_{n}^{(1)}-\alpha_{n}^{(0)})+(n-4) x_{n-4}(\alpha_{n}^{(2)}-\alpha_{n}^{(1)}+\alpha_{n}^{(0)})+\\+
x_{n-2i}(n-2i)(\alpha_n^{(i)}-\alpha_n^{(i-1)}+\alpha_n^{(i-2)}+\cdots+(-1)^i \alpha_n^{(0)})+\ldots+
\end{gather*}
З іншого боку,  має виконуватися тотожність,
$$
D_\mathcal{F}(\psi(x_n))=\psi(\mathcal{D}(x_n))
=n \psi(x_{n-1})=n \sum_{i=0}^{n-1}\alpha_{n-1}^{(i)}x_i.
$$
Таким чином, отримуємо  такі рекурентні  співвідношення на  $\alpha^{(i)}_n:$
\begin{gather*}
n\alpha_{n}^{(0)}=n \alpha_{n-1}^{(0)},\\
(n-2) (\alpha_{n}^{(1)}-\alpha_{n}^{(0)})=n \alpha^{(1)}_{n-1}, \alpha^{(1)}_{2}=0,\\
(n-4)(\alpha_{n}^{(2)}-\alpha_{n}^{(1)}+\alpha_{n}^{(0)})=n \alpha^{(2)}_{n-1},\alpha_4^{(1)}=0,\\
\ldots\\
(n-2i)\sum_{n-2i>0}  \alpha^{(i)}_n (-1)^i =n \alpha^{(i)}_{n-1}, \alpha_{2i}^{(i)}=0.
\end{gather*}
Після спрощення  отримуємо таку  систему рівнянь 
$$
\left\{ 
\begin{array}{l}
\alpha_{n}^{(0)}=1 ,\\
\alpha_{n}^{(1)}=n \left( \dfrac{\alpha_{n-1}^{(1)}}{n-2}+ \dfrac{\alpha_{n-1}^{(0)}}{n} \right),\\
\alpha_{n}^{(2)}=n \left( \dfrac{\alpha_{n-1}^{(2)}}{n-4}+ \dfrac{\alpha_{n-1}^{(1)}}{n-2} \right),\\
\ldots \\
\alpha_{n}^{(s)}=n \left( \dfrac{\alpha_{n-1}^{(2)}}{n-2s}+ \dfrac{\alpha_{n-1}^{(1)}}{n-2(s-1)} \right),\\
\ldots
\end{array}
\right.
$$
\end{proof}

\begin{lemma}Рекурентне рівняння
$$
x_n=n \left( \frac{x_{n-1}}{n-s}+ \frac{g_{n-1}}{n-(s-2)}\right), x(s)=0,
$$
має наступний розв'язок
$$
x_n=n^{\underline{s}} \sum_{i=s}^{n-1} \frac{g_i}{i^{\phantom{}^{\underline{s-1}}}\,\, (i-(s-3))}.
$$
\end{lemma}

\begin{proof}
Маємо
\begin{gather*}
n \left( \frac{x_{n-1}}{n-s}+ \frac{g_{n-1}}{n-(s-2)}\right)=n \left( \frac{(n-1)^{\underline{s}}}{n-s} \sum_{i=s}^{n-2} \frac{g_i}{i^{\phantom{}^{\underline{s-1}}}\,\, (i-(s-3))}+ \frac{g_{n-1}}{n-(s-2)}\right)=\\=
n \left( \frac{(n-1)^{\underline{s}}}{n-s} \sum_{i=s}^{n-2} \frac{g_i}{i^{\phantom{}^{\underline{s-1}}}\,\, (i-(s-3))}+ \frac{(n-1)^{\underline{s-1}}\,g_{n-1}}{((n-1)^{\underline{s-1}})((n-1)-(s-3))}\right)=\\=
\frac{n (n-1)^{\underline{s}}}{n-s} \sum_{i=s}^{n-2} \frac{g_i}{i^{\phantom{}^{\underline{s-1}}}\,\, (i-(s-3))}+ \frac{n (n-1)^{\underline{s-1}}\,g_{n-1}}{((n-1)^{\underline{s-1}})((n-1)-(s-3))}=\\=
 n^{\underline{s}}\left( \sum_{i=s}^{n-2} \frac{g_i}{i^{\phantom{}^{\underline{s-1}}}\,\, (i-(s-3))}+ \frac{g_{n-1}}{((n-1)^{\underline{s-1}})((n-1)-(s-3))}\right)=\\=
 n^{\underline{s}}  \sum_{i=s}^{n-1} \frac{g_i}{i^{\phantom{}^{\underline{s-1}}}\,\, (i-(s-3))}.
\end{gather*}
Тут використали співвідношення 
$$
\frac{n (n-1)^{\underline{s}}}{n-1}=n^{\underline{s}}, n (n-1)^{\underline{s-1}}=n^{\underline{s}}, 
$$
\end{proof}

Розв'язуючи  рекурентні  рівняння при початковій  умові $\alpha_{2s}^{(s)}=0$ послідовно  знаходимо
\begin{gather*}
\alpha_{n}^{(0)}=1,\\
\alpha_{n}^{(1)}=\frac{1}{2} (n-1)(n-2),\\
\alpha_{n}^{(2)}=\frac16\, \left( n-4 \right)  \left( n-3 \right)  \left( n-2 \right) n,\\
\alpha_{n}^{(3)}={\frac {1}{144}}\, \left( n-1 \right) n \left( n-4 \right)  \left( n-5 \right)  \left( 7\,n-17 \right)  \left( n-6 \right),\\
\alpha_{n}^{(4)}={\frac {1}{2880}}\, \left( n-8 \right)  \left( 39\,{n}^{2}-296\,n+545 \right)  \left( n-7 \right)  \left( n-6 \right)  \left( n-2 \right)  \left( n-1 \right) n
\end{gather*}

Будемо розглядати невідомі  послідовності $\alpha_n^{(s)}, s>1$ у  базисі із спадних факторіалів.  Нехай 
$$
\alpha_n^{(s)}=(n-(2s-1))\left(\beta_0^{(s)} n^{\underline{s-1}}+\beta_1^{(s)} n^{\underline{s}}+\cdots+\beta_{s}^{(s)} n^{\underline{2s-1}}\right)=(n-(2s-1))\sum_{i=0}^{s}\beta_i^{(s)} n^{\underline{s-1+i}}.
$$

Тоді
\begin{gather*}
\alpha_n^{(s)}=(n-(2s-1))\sum_{i=0}^{s}\beta_i^{(s)} n^{\underline{s-1+i}}=\sum_{i=0}^{s-1}\beta_i^{(s)}(n-(s+i-1)-(s-i)) n^{\underline{s+i}}=\\=\sum_{i=0}^{s}\beta_i^{(s)}(n-(s+i-1)) n^{\underline{s+i-1}}\,-\sum_{i=0}^{s}\beta_i^{(s)}(s-i) n^{\underline{s-1+i}}=
\sum_{i=0}^{s}\beta_{i}^{(s)}n^{\phantom{}^{\underline{s+i}}}-\sum_{i=0}^{s-1}(s-i)\beta_{i}^{(s)}n^{\phantom{}^{\underline{s-1+i}}}
\end{gather*}
З іншого боку ми маємо
\begin{align*}
&\alpha_{n-1}^{(s)}=(n-2s)\sum_{i=0}^{s}\beta_i^{(s)} (n-1)^{\underline{s-1+i}}\,,\\
&\alpha_{n-1}^{(s-1)}=(n-(2s-2))\sum_{i=0}^{s-1}\beta_i^{(s-1)} (n-1)^{\underline{s+i-2}}\,.
\end{align*}
Тоді
\begin{gather*}
\alpha_{n}^{(s)}=n\left(\frac{\alpha_{n-1}^{(s)}}{n-2s}+\frac{\alpha_{n-1}^{(s-1)}}{n-(2s-2)}\right)=\sum_{i=0}^{s} n \beta_{i}^{(s)}(n-1)^{\phantom{}^{\underline{s-1+i}}}+\sum_{i=0}^{s-1} n \beta_{i}^{(s-1)}(n-1)^{\phantom{}^{\underline{s-2+i}}}=\\=\sum_{i=0}^{s}  \beta_{i}^{(s)}n^{\phantom{}^{\underline{s+i}}}+\sum_{i=0}^{s-1}  \beta_{i}^{(s-1)}n^{\phantom{}^{\underline{s-1+i}}}.
\end{gather*}
Отже 
$$
\sum_{i=0}^{s-1}  \beta_{i}^{(s-1)}n^{\phantom{}^{\underline{s-1+i}}}=-\sum_{i=0}^{s-1}(s-i)\beta_{i}^{(s)}n^{\phantom{}^{\underline{s-1+i}}}
$$
Звідси, прирівнявши відповіднi  коефіцієнти, зразу знаходимо, що

$$
\beta_i^{(s)}=-\frac{\beta_{i}^{(s-1)}}{s-i},i=0\ldots s-1.
$$

Коефіцієнт $\beta_{s}^{(s)}$  знайдемо з умови, що $\alpha_{2s}^{(s)}=0.$ Маємо
$$
\alpha_{2s}^{(s)}=\sum_{i=0}^{s}\beta_i^{(s)}(2s)^{\phantom{}^{\underline{s-1+i}}}=\sum_{i=0}^{s-1}\beta_i^{(s)}(2s)^{\phantom{}^{\underline{s-1+i}}}+\beta_{s}^{(s)} (2s)^{\phantom{}^{\underline{2s-1}}}=0
$$
Звідси, врахувавши $(2s)^{\phantom{}^{\underline{2s-1}}}=2s(2s-1)\ldots 2=(2s)!,$ отримаємо 
$$
\beta_{s}^{(s)}=-\frac{1}{(2s)!}\sum_{i=0}^{s-1}\beta_i^{(s)}(2s)^{\phantom{}^{\underline{s-1+i}}}=-\sum_{i=0}^{s-1} \frac{\beta_{i}^{(s)}}{ (s-i+1)!}.
$$
Таким чином, маємо  такi рекурентнi співвідношення  для $\beta_{n}^{(s)}$:
\begin{align*}
&\beta_0^{(s)}=-\frac{\beta_0^{(s-1)}}{s},\\
&\beta_1^{(s)}=-\frac{\beta_1^{(s-1)}}{s-1},\\
&\beta_2^{(s)}=-\frac{\beta_1^{(s-1)}}{s-2},\\
&\ldots \\
&\beta_{s-1}^{(s)}=-\beta_1^{(s-1)},\\
&\beta_{s}^{(s)}:=b_s=-\sum_{i=0}^{s-1} \frac{\beta_{i}^{(s)}}{ (s-i+1)!}.
\end{align*}

Звідси знаходимо, що 
$$
\beta_i^{s}=\frac{(-1)^{s-i}}{(s-i)!} b_i, \text{  for $i=0,\ldots,s-1$. }
$$
Oтримуємо таке рекурентне співвідношення на послідовність $b_s:$
$$
b_s=-\frac{(-1)^{s}}{((s+1)!)^2}-\sum_{i=1}^{s-1} \frac{(-1)^{s-i}}{((s-i)(s-i+1)!)} b_i,
$$
Поклавши $b_0=1$ отримаємо, що елементи послідовності $b_n$ задовільняють таке рекурентне співвідношення
$$
\sum_{i=0}^s \frac{(-1)^{n-i}}{(s-i)!(s-i+1)!}b_n=0.
$$
Розглянемо ряд 
$$
\sum_{n=0}^{\infty} \frac{(-1)^{n}}{n!(n+1)!}z^n=\frac{1}{\sqrt{z}} J_1(\sqrt{4z}\,)
$$
і  звичайну породжуючу функцію 
$$
G(b_n,z)=\sum_{n=0}^{\infty} b_n z^n.
$$
Тоді, враховуючи рекурентне співвідношення знйдемо, що $$ G(b_n,z) \frac{J_1(\sqrt{4z}\,)}{\sqrt{z}}=1.$$
Тому
$$
G(b_n,z)=\frac{\sqrt{z}}{J_1(\sqrt{4z}\,)}=1+\frac12\,z+\frac16\,{z}^{2}+{\frac {7}{144}}\,{z}^{3}+{\frac {13}{960}}\,{z}
^{4}+{\frac {107}{28800}}\,{z}^{5}+{\frac {409}{403200}}\,{z}^{6}+\cdots+
$$

Отже   ми  довели таку  теорему

\begin{theorem}
A Appel-Fibbonaci  $\psi_{AF}$  interwining map has  the form 
$$
\psi_{AL}(x_n)= x_{n+1}+\alpha^{(1)}_n x_{n-1}+\alpha^{(2)}_n x_{n-3}+\ldots+\alpha^{(i)}_n x_{n+1-2i}+\ldots+\alpha_n^{\left( \left[ \frac{n-1}{2} \right]\right)} x_{n+1-2 \left[ \frac{n-1}{2} \right]},
$$
where
$$
\alpha^{(s)}_n=(n-2s+1)\left( \frac{(-1)^{s}}{s!} b_0 n^{{\underline{s-1}}}+\cdots+\frac{(-1)^{s-i}}{(s-i)!}\, b_i n^{\underline{s+i}}+\cdots+ b_s n^{\underline{2s-1}} \right),
$$
and the generating function for  $b_0, b_1, \ldots, b_n,\ldots$  are defined через обернену  функцію Бесселя
$$
\sum_{i=0}^{\infty}b_i z^i=\frac{\sqrt{z}}{J_1(\sqrt{4z}\,)}.
$$
\end{theorem}

\textbf{Задача.} Знайти формулу для $b_n.$

\newpage
Із рекурентного співвідношення для  многочленів Фібоначі виплаває такйи  зв'язок між диференціюваннями  Вейтценбека і Фібоначчі
$$
\mathcal{D}(x_n)=2 D_F(x_n)-x_2 D_F(x_{n-1}).
$$

\end{document}